\documentclass{amsart}
\usepackage{ucs}
\usepackage[utf8x]{inputenc}
\usepackage{verbatim}
\usepackage{paralist}
\usepackage{leftidx}
\usepackage{dsfont}
\usepackage{amsthm}
\usepackage{amsmath}
\usepackage{amssymb}
\usepackage{amscd}
\usepackage{graphicx}
\usepackage{setspace}
\usepackage[all]{xy}
\usepackage{mathrsfs} 
\usepackage{hyperref}
\usepackage{mathtools}
\title[An alternative proof of the Faltings--Elkies bound]{An alternative proof of the Faltings--Elkies bound}
\author{Robert Wilms}
\address{Robert Wilms\\
	Laboratoire de Mathématiques Nicolas Oresme\\
	Université de Caen--Normandie BP 5186\\
	14032 Caen Cedex\\
	France}
\email{robert.wilms@unicaen.fr}
\thanks{The author gratefully acknowledges support from the Swiss National Science Foundation grant ``Diophantine Equations: Special Points, Integrality, and Beyond'' (n$^\circ$ 200020\_184623)}
\date{\today}
\begin{document}
	\numberwithin{equation}{section}
	\newtheorem{Def}{Definition}
	\numberwithin{Def}{section}
	\newtheorem{Rem}[Def]{Remark}
	\newtheorem{Lem}[Def]{Lemma}
	\newtheorem{Que}[Def]{Question}
	\newtheorem{Cor}[Def]{Corollary}
	\newtheorem{Exam}[Def]{Example}
	\newtheorem{Thm}[Def]{Theorem}
	\newtheorem*{clm}{Claim}
	\newtheorem{Pro}[Def]{Proposition}
	\newcommand\gf[2]{\genfrac{}{}{0pt}{}{#1}{#2}}
	
\begin{abstract}
We give an alternative proof of the Faltings--Elkies bound on the average value of the Arakelov--Green function in pairs of a given set of $n$ points on a Riemann surface, which grows asymptotically like $O((\log n)/n)$.
Our result is effective in terms of bounds of the Arakelov--Green function with respect to a given covering by local coordinates.
\end{abstract}
\maketitle
\section{Introduction}
In 1974 Arakelov \cite{Ara74} introduced an intersection theory on arithmetic surfaces motivated by the applications of intersection theory on algebraic surfaces in the proofs of function field analogues of arithmetic conjectures, for example of the Mordell conjecture by Parshin \cite{Par68}. In particular, he defined the now called Arakelov--Green function $G(\cdot,\cdot)$ on the self-product $X^2$ of a compact and connected Riemann surface $X$ as the archimedean contribution for the intersection pairing. Faltings \cite{Fal84} largely expanded this theory by proving analogues of the classical Riemann--Roch formula, the Hodge index theorem, the Noether formula and the non-negativity of the self-intersection number $\omega^2$ of the relative dualizing sheaf. 

In his proof of the non-negativity of $\omega^2$, Faltings showed that the average value of the logarithm of the Arakelov--Green function $g(\cdot,\cdot)=\log G(\cdot, \cdot)$ in the pairs of different points of a given set of $n$ different points on a Riemann surface tends to a non-positive value for increasing $n$. Elkies \cite[VI, §5]{Lan88} improved this result by showing that this average is asymptotically bounded by $(\log n)/(8\pi n)$. For $g=1$ the constants in Elkies' result can be made explicit in terms of the $j$-invariant, as worked out by Baker and Petsche \cite[Proposition 4]{BP05}. The aim of this note is to give an alternative proof of Elkies' result and to make sense of the involved constants.

We write $\mathbb{D}_r=\{z\in\mathbb{C}~|~|z|<r\}$ for the open disc around the origin in $\mathbb{C}$ of radius $r$. For any compact and connected Riemann surface of genus $g\ge 1$ we write $\nu$ for the canonical translation-invariant $(1,1)$-form on its Jacobian $J=\mathrm{Pic}^0(X)$, and we set $\mu=\frac{1}{g}I^*\nu$, where $I\colon X\to J$ denotes the Abel--Jacobi embedding corresponding to any base point. Let $\|\theta\|\colon J\to \mathbb{R}_{\ge 0}$ denote the norm of the Riemann $\theta$-function on $J$ and set $H(X)=\int_{J}\log\|\theta\| \tfrac{\nu^g}{g!}$.
We refer to Section \ref{sec_riemannsurfaces} for more details on these definitions. Our main result is the following theorem.
\begin{Thm}\label{thm_elkies}
	Let $X$ be a compact and connected Riemann surface of genus $g\ge 1$. Further, let $\{z_j\colon U_j\to \mathbb{C}\}_{j\in J}$ be a family of local coordinate functions on $X$. We assume that there exist positive numbers $r_2>r_1>0$ with $r_2-r_1\le 1$ such that $X=\bigcup_{j=1}^m z_j^{-1}(\mathbb{D}_{r_1})$ and $z_j(U_j)=\mathbb{D}_{r_2}$ for all $j\in\{1,\dots,m\}$.
	We set
	$$C_0=\sup_{j\in J}\sup_{P,Q\in U_j, P\neq Q}\left|\log|z_j(P)-z_j(Q)|-g(P,Q)\right|,$$
	where $g(\cdot,\cdot)$ denotes the logarithm of the Arakelov--Green function of $X$.
	Let $C_1>0$ and $C_2>0$ be positive real numbers such that
	$$C_2idz_jd\overline{z}_j\le\mu\le C_1idz_j d\overline{z}_j$$
	on $U_j$ for every $j\in J$ and $C_2\le \frac{e^{4C_0}}{2\pi}$.
	For any $n\ge 1$ and any pairwise different points $x_1,\dots,x_n\in X$ it holds
	\begin{align*}
		&\sum_{j<k}^n g(x_j,x_k)\\
		&\le n\left(\left(\frac{C_1e^{4C_0}}{2C_2}+1\right)(\log n+1)+\tfrac{3}{2}g\log g+3-\tfrac{g+1}{g}H(X)+C_0-\log(r_2-r_1)\right).
	\end{align*}
\end{Thm}

Our proof is completely different to Faltings' and Elkies' method, which uses the convolution of the heat kernel. Instead, our proof starts where Faltings applies his bound of the average. We will consider the line bundles
$$\mathcal{M}=\mathcal{L}\otimes \mathcal{O}_X\left(\sum_{j=1}^n y_j\right) \quad \text{and}\quad \mathcal{N}=\mathcal{M}\otimes\mathcal{O}_X\left(-\sum_{j=1}^n x_j\right)$$
for a general admissible metrized line bundle $\mathcal{L}$ of degree $g-1$ on $X$ and general points $x_1,\dots,x_n\in X$ and $y_1,\dots,y_n\in X$. We will compute the Faltings volume of a certain parallelotope in $H^0(X,\mathcal{M})$ in two different ways: First for the general situation and second in the special situation where $x_j=y_j$ for all $j$. This will give us an explicit equation for $\sum_{j\neq k}^n g(x_j,x_k)$ in terms of the norm of the $\theta$-functions $\|\theta\|$ in $\mathcal{L}$ and $\mathcal{N}$ and a determinant, which can be estimated by Hadamard's inequality by further Arakelov--Green functions and $\theta$-functions. Integrating with respect to $\mathcal{L}$ and to the points $y_1,\dots,y_n$ we will reduce the statement of the theorem to upper bounds of $\|\theta\|$ and of the integral
$$A_n(x)=-\int_{X^n}\min_{1\le j\le n} g(y_j,x)\mu(y_1)\dots\mu(y_n).$$
While upper bounds of $\|\theta\|$ are already known, we compute a bound of $A_n(x)$ using the local property $g(y,x)=\log|z(y)-z(x)|+C^{\infty}$ of the Arakelov--Green function for a local coordinate function $z$.

The constant $C_0$ has been bounded more explicitly by Merkl \cite{Mer11} with some improvements by Bruin \cite{Jav14}. Let us recall from \cite[Definition 3.1.1]{Jav14} that a \emph{Merkl atlas} on a compact and connected Riemann surface $X$ of genus $g\ge 1$ is a quadruple $(\{(U_j,z_j)\}_{j=1}^m,r_1,M, C_1)$ where $\{(U_j,z_j)\}_{j=1}^m$ is an atlas of $X$ and $\frac{1}{2}<r_1<1$, $M\ge 1$, and $C_1>0$ are real numbers such that
\begin{enumerate}[(i)]
	\item $z_j(U_j)=\mathbb{D}_1$ for all $j\in \{1,\dots, m\}$,
	\item $\bigcup_{j=1}^m z_j^{-1}(\mathbb{D}_{r_1})=X$,
	\item $\sup_{x\in U_j\cap U_k}|dz_j(x)/dz_k(x)|\le M$ for all $j,k\in\{1,\dots,m\}$,
	\item $\mu\le C_1idz_jd\overline{z}_j$ for all $j\in\{1,\dots,m\}$.
\end{enumerate}
Merkl and Bruin \cite[Theorem 3.1.2]{Jav14} proved that if $(\{(U_j,z_j)\}_{j=1}^m,r_1,M, C_1)$ is a Merkl atlas, then
\begin{align*}
	&\max_{1\le j\le m}\sup_{P,Q\in U_j, P\neq Q}\left|\log|z_j(P)-z_j(Q)|-g(P,Q)\right|\\
	&\le -\frac{330m}{(1-r_1)^{3/2}}\log(1-r_1)+13.2mC_1+(m-1)\log M.\nonumber
\end{align*}
Let us call $(\{(U_j,z_j)\}_{j=1}^m,r_1,M, C_1, C_2)$ an \emph{extended Merkl atlas} if the quadruple $(\{(U_j,z_j)\}_{j=1}^m,r_1,M, C_1)$ is a Merkl atlas and $C_2>0$ is a real number such that
\begin{enumerate}
	\item[(v)] $\mu\ge C_2idz_jd\overline{z}_j$ for all $j\in\{1,\dots,m\}$.
\end{enumerate}
This implies $C_2\le \frac{1}{2\pi}$ as $\int_{\mathbb{D}_1}idzd\overline{z}=2\pi$ and $\int_{X}\mu=1$.
If we combine Theorem \ref{thm_elkies} with the result by Merkl and Bruin, we get the following corollary.
\begin{Cor}\label{cor_elkies}
	Let $X$ be a compact and connected Riemann surface of genus $g\ge 1$ and $(\{(U_j,z_j)\}_{j=1}^m,r_1,M, C_1, C_2)$ an extended Merkl atlas on $X$. For any $n\ge 1$ and any pairwise different points $x_1,\dots,x_n\in X$ it holds
	\begin{align*}
		&\sum_{j<k}^n g(x_j,x_k)\\
		&\le n\left(\left(\frac{C_1e^{4C_0}}{2C_2}+1\right)(\log n+1)+\tfrac{3}{2}g\log g+3-\tfrac{g+1}{g}H(X)+C_0-\log(1-r_1)\right).
	\end{align*}
	with $C_0=-\frac{330m}{(1-r_1)^{3/2}}\log(1-r_1)+13.2mC_1+(m-1)\log M$.
\end{Cor}

In a recent work, Looper, Silverman, and the author \cite{LSW21} applied the analogue of the Faltings--Elkies bound for polarized metrized graphs by Baker--Rumely \cite[Proposition 13.7]{BR07} to prove an explicit version of the Bogomolov conjecture over function fields. The technique of proof would also apply to the number field case if explicit constants for the Faltings--Elkies bound were known. In \cite{EdJ11} Edixhoven and de Jong computed a Merkl atlas and the asymptotics of the constants $m$, $r_1$, $M$ and $C_1$ for the modular curve $X_1(pl)$ for two different primes $p$ and $l$. Thus, by Theorem \ref{thm_elkies} and the methods in \cite{LSW21} an asymptotic bound for $C_2^{-1}$ would lead to an asymptotic bound for the number of small height of $X_1(pl)$.

\subsection*{Outline}
In Section \ref{sec_riemannsurfaces} we recall the definitions of the Arakelov--Green function $G$, of the Riemann $\theta$-function and its norm $\|\theta\|$, and of the invariant $H(X)$. In Section \ref{sec_metrics} we deduce a relation between $G$ and $\|\theta\|$ by studying canonical metrics on the determinant of cohomology. Finally, we prove Theorem \ref{thm_elkies} in Section \ref{sec_greenbound}. 

\section{Arakelov invariants on Riemann surfaces}\label{sec_riemannsurfaces}
In this section we give the definitions of some invariants associated to Riemann surfaces motivated by Arakelov theory. For details we refer to \cite{Wil17}.
Let $X$ be any compact and connected Riemann surface of genus $g\ge 1$.
First, we choose a symplectic basis of homology $A_1,\dots,A_g,B_1,\dots B_g\in H_1(X,\mathbb{Z})$ and a basis of one forms $\omega_1,\dots,\omega_g\in H^0(X,\Omega_X^1)$ such that $\int_{A_j}\omega_i=\delta_{ij}$ for the Kronecker symbol $\delta_{ij}$. The entries of the period matrix $\Omega$ of $X$ are given by $\Omega_{ij}=\int_{B_i}\omega_j$. We denote its imaginary part by $Y=\mathrm{Im}~\Omega$. The Jacobian variety $J$ of $X$ is defined by $J=\mathbb{C}^g/(\mathbb{Z}^g+\Omega\mathbb{Z}^g)$.
For a base point $x_0\in X$ we denote the Abel--Jacobi embedding $I_{x_0}\colon X\to J,~x\mapsto \left(\int_{x_0}^x\omega_1,\dots,\int_{x_0}^x\omega_g\right)$.

On $J$ we have the canonical $(1,1)$ form 
$$\nu=\tfrac{i}{2}\sum_{j,k=1}^g\left(Y^{-1}\right)_{jk}dZ_j\wedge d\overline{Z}_k,$$
where $Z_1,\dots, Z_g$ are coordinates in $\mathbb{C}^g$. The form $\nu$ is translation-invariant, positive, and its associated volume form $\nu^g$ has volume $\int_J\nu^g=g!$. Further, it induces the canonical form $\mu=\frac{1}{g}I_{x_0}^*\nu$ on $X$, which is independent of $x_0$ and has volume $\int_X\mu=1$. Note that $\mu$ is also a positive form. The \emph{Arakelov--Green function} is the unique function $G\colon X^2\to \mathbb{R}_{\ge 0}$, such that $G^2$ is $C^{\infty}$ and it holds 
\begin{align}\label{equ_green}
	\partial_x\overline{\partial}_x\log G(x,y)^2=2\pi i(\mu(x)-\delta_y(x))\quad  \text{and}\quad \int_X\log G(x,y)\mu(x)=0.
\end{align}
We write $g(x,y)=\log G(x,y)$ for its logarithm. We remark that $g(x,y)$ is superharmonic as a function of $x$, as $\mu$ is positive.
We define the norm of the Riemann $\theta$-function $\|\theta\|\colon J\to \mathbb{R}_{\ge 0}$ by
$$\|\theta\|(z)=\det(Y)^{1/4}\exp(-\pi\ltrans{y}Y^{-1}y)\left|\sum_{n\in\mathbb{Z}^g}\exp(\pi\ltrans{n}\Omega n+2\pi\ltrans{n}z)\right|,$$
where we think of $z\in J$ as a vector in $\mathbb{C}^g$ representing $z$ and we set $y=\mathrm{Im}~z$. 
We denote $J_{g-1}=\mathrm{Pic}^{g-1}(X)$ for the variety of line bundles on $X$ of degree $g-1$ and $\Theta\subseteq J_{g-1}$ for the divisor consisting of the line bundles $\mathcal{L}$, such that $H^0(X,\mathcal{L})\neq 0$. Then there exists a theta characteristic $\kappa$, such that $\|\theta\|(\mathcal{L}-\kappa)=0$ if and only if $\mathcal{L}\in\Theta$. In the following we shortly write $\|\theta\|(\mathcal{L})$ for $\|\theta\|(\mathcal{L}-\kappa)$ for any line bundle $\mathcal{L}\in J_{g-1}$. Moreover, we set $\|\theta\|(D):=\|\theta\|(\mathcal{O}_X(D))$ if $D$ is any divisor of degree $0$ or $g-1$ on $X$. Furthermore, we also denote $\nu$ for the canonical 1-1 form on $J_{g-1}$ obtained by the pullback of $\nu$ along the translation $J_{g-1}\to J,~\mathcal{L}\mapsto\mathcal{L}-\kappa$.

We define $H(X)=\int_{J}\log\|\theta\|(z)\tfrac{\nu^g(z)}{g!}$.
As $\nu$ is translation-invariant, we obtain
$$H(X)=\int_{J_{g-1}} \log\|\theta\|(z+z_0)\tfrac{\nu^g(z)}{g!}$$
for any $z_0\in J$.

Finally, we recall bounds for $\|\theta\|$. For any $\rho>0$ we define the number
$$c_{g,\rho}=\log \tfrac{g+2}{2}+\tfrac{g\epsilon_g}{2}\log\tfrac{g+2}{\pi\sqrt{3}}+\tfrac{g(1+\rho)}{4}\log\tfrac{1+\rho}{2\rho},$$
where $\epsilon_g=0$ for $g\le 3$ and $\epsilon_g=1$ for $g\ge 4$. Based on Autissier's bound \cite[Proposition A.1]{Par18} of the $\theta$-function,  
we determined the bound
\begin{align}\label{bound_theta}
	\log\|\theta\|(z)\le -\rho H(X)+c_{g,\rho}
\end{align}
in \cite[Corollary 5.7]{Wil17}. A direct computation shows that $c_{g,1/g}\le\frac{3}{2}g\log g+4$ for all $g\ge 1$.

\section{Admissible metrics}\label{sec_metrics}
In this section we discuss metrics on line bundles on Riemann surfaces. This is based on Faltings work in \cite[Section 3]{Fal84}, see also \cite[Chapter VI.]{Lan88} for details. We will also deduce a new formula relating the Arakelov--Green function $G$ and the norm function $\|\theta\|$, which is motivated by a similar formula by Fay. We continue the notation from the previous section.

Let $\mathcal{L}$ be a line bundle on $X$. By a \emph{hermitian metric} on $\mathcal{L}$ we mean a collection of hermitian metrics $\|\cdot\|(x)$ on $\mathcal{L}|_x$ for every point $x\in X$, such that for any open subset $U\subseteq X$ and any section $s\in H^0(U,\mathcal{L})$ the function
$$\|s\|^2\colon U\to \mathbb{R}_{\ge 0},\quad x\mapsto \|s\|(x)^2$$
is $C^\infty$ on $U$. We call a hermitian metric \emph{admissible} if 
$$\partial\overline{\partial}\log\|s\|^2=2\pi i(\deg\mathcal{L})\mu,$$
on any open subset $U\subseteq X$ and for any invertible section $s\in H^0(U,\mathcal{L})$.
Two different admissible hermitian metrics on a line bundle differ by a harmonic function on $X$ and hence, by a constant factor.
We call a line bundle equipped with an admissible hermitian metric an \emph{admissible line bundle}.
If $\mathcal{L}=\mathcal{O}_X(D)$ is the line bundle induced by an effective divisor $D=\sum_{j=1}^n x_j$ on $X$, we obtain an admissible hermitian metric on $\mathcal{L}$ by setting
$$\|1\|(y)=\prod_{j=1}^n G(x_j,y).$$
Taking tensor products, we also obtain an admissible hermitian metric on $\mathcal{O}_X(D)$ for any divisor $D$ on $X$.

Let $\mathcal{L}$ be any admissible line bundle. Faltings introduced in \cite[Theorem 1]{Fal84} a canonical metric on the determinant of cohomology
$$\lambda(\textbf{R}\Gamma(X,\mathcal{L}))=\bigwedge^{\max}H^0(X,\mathcal{L})\otimes \bigwedge^{\max} H^1(X,\mathcal{L})^{\otimes -1}$$
which is given up to a scalar factor independent of $\mathcal{L}$, which can be fixed by requiring that $\lambda(\textbf{R}\Gamma(X,\Omega_X^1))=\bigwedge^g H^0(X,\Omega_X^1)$ is equipped with the metric induced by the inner product $\langle\omega,\omega'\rangle=\frac{i}{2}\int_X\omega\wedge\omega'$ on $H^0(X,\Omega_X^1)$. If $\mathcal{L}$ is a line bundle of degree $\deg\mathcal{L}=g-1$ satisfying $H^0(X,\mathcal{L})=0$, then we have canonically $\lambda(\textbf{R}\Gamma(X,\mathcal{L}))\cong \mathbb{C}$ and its canonical norm is given by the number $\|\theta\|(\mathcal{L})^{-1}\cdot \exp(-\delta(X)/8)$ for some constant $\delta(X)$ only depending on $X$, see \cite[p.~402]{Fal84}.

Let $\mathcal{L}$ be a line bundle of degree $\deg \mathcal{L}=g-1$ on $X$, such that $H^0(X,\mathcal{L})=0$ and let $y_1,\dots, y_n\in X$ be pairwise different points. We set $$\mathcal{M}=\mathcal{L}\otimes\mathcal{O}_X\left(\sum_{j=1}^n y_j\right).$$
We would like to define an admissible metric on $\mathcal{L}$ and hence, also on $\mathcal{M}$.
Let $s'_1\in H^0(X,\mathcal{L}\otimes\mathcal{O}_X(y_1))$ be a non-zero global section. As $s'_1$ has the same divisor as $\|\theta\|(\mathcal{L}+y_1-x)$ as a function of $x\in X$, we may define an admissible metric on $\mathcal{L}\otimes\mathcal{O}_X(y_1)$ by setting $\|s'_1\|(x)=\|\theta\|(\mathcal{L}+y_1-x)$. This metric is admissible since $\partial\overline{\partial}\log\|\theta\|^2=2\pi i(\nu-\delta_\Theta)$.
If we define $U_i=X\setminus \{y_i\}$, we obtain an induced section $s_1\in H^0(U_1,\mathcal{L})$ and the metric on $\mathcal{L}$ induced by the metric on $\mathcal{L}\otimes\mathcal{O}_X(y_1)$ is given by
$$\|s_1\|(x)=\frac{\|\theta\|(\mathcal{L}+y_1-x)}{G(y_1,x)}.$$
As all admissible metrics differ only by a constant factor, we may also choose sections $s'_j\in H^0(X,\mathcal{L}\otimes\mathcal{O}_X(y_j))$ for $2\le j\le n$, such that we have
$$\|s_j\|(x)=\frac{\|\theta\|(\mathcal{L}+y_j-x)}{G(y_j,x)}$$
for the induced section $s_j\in H^0(U_j,\mathcal{L})$. We take on $\mathcal{M}$ the admissible hermitian metric induced by the product of the admissible hermitian metrics on its factors. By multiplying with the constant section $1$, we may consider $t_j=s'_j\cdot 1\in H^0(X,\mathcal{M})$ also as a section of the admissible line bundle $\mathcal{M}$. The norm of the $t_j$'s are given by
$$\|t_j\|(x)=\|\theta\|(\mathcal{L}+y_j-x)\cdot\prod_{k\neq j}^n G(y_k,x).$$
As the $y_j$'s are pairwise different, it holds $t_j(y_k)\neq 0$ if and only if $j=k$.
Hence, the $t_j$'s form a basis of $H^0(X,\mathcal{M})$.

The following construction is motivated by the construction in the proof of \cite[Theorem 2]{Fal84}.
It follows from the Riemann--Roch theorem, that $H^1(X,\mathcal{L})=0$ and hence, also $H^1(X,\mathcal{M})=0$. Thus, the canonical metric on $\lambda(\textbf{R}\Gamma(X,\mathcal{M}))$ defines a volume form on $H^0(X,\mathcal{M})$. For points $x_1,\dots,x_n\in X$ in general position it holds $H^0(X,\mathcal{N})=0$, where $\mathcal{N}=\mathcal{M}\otimes\mathcal{O}_X\left(-\sum_{j=1}^n x_j\right)$. The volume form on $\lambda(\textbf{R}\Gamma(X,\mathcal{N}))\cong \mathbb{C}$ is given by $\|\theta\|(\mathcal{N})^{-2}\cdot \exp(-\delta(X)/4)$. Note that we get squares for the volume forms as we working over the complex numbers, which are two-dimensional over the real numbers. If the points $x_1,\dots,x_n$ are moreover pairwise different, there is an isomorphism 
\begin{align*}
	\phi\colon H^0(X,\mathcal{M})\cong ~\bigoplus_{j=1}^n\mathcal{M}[x_j].
\end{align*}
We equip $\mathcal{M}[x_j]$ with the hermitian metric induced by the hermitian metric on $\mathcal{M}$ and $\bigoplus_{j=1}^n\mathcal{M}[x_j]$ with the induced hermitian metric for orthogonal sums. As shown in \cite[Lemma V.3.3]{Lan88}, the isomorphism $\phi$ alters volumes by the factor
$$\prod_{j\neq k}^n G(x_j,x_k)\cdot \|\theta\|(\mathcal{N})^2\cdot \exp(\delta(X)/4).$$
Let $f_1,\dots,f_n\in \bigoplus_{j=1}^n\mathcal{M}[x_j]$ be an orthonormal basis and write $A=(a_{jk})\in \mathbb{C}^{n\times n}$ for the base change matrix, such that $\phi(t_j)=\sum_{k=1}^n a_{jk} f_k$. We set $\|\det(t_j(x_k))\|=|\det(A)|$.

Let $P\subseteq H^0(X,\mathcal{M})$ be the parallelotope spanned by $t_1,\dots,t_n, it_1,\dots,it_n$ in $H^0(X,\mathcal{M})$. As the $a_{jk}$'s are the coordinates of the images $\phi(t_j)$ of the $t_j$'s under the isomorpism $\phi$ for a orthonormal basis, the volume of $\phi(P)$ can be computed by
\begin{align*}
	V(\phi(P))&=\left|\det\begin{pmatrix} (\mathrm{Re}~a_{jk}) & (-\mathrm{Im}~a_{jk})\\ (\mathrm{Im}~a_{jk}) & (\mathrm{Re}~a_{jk})\end{pmatrix}\right|\\
	&=|\det(\mathrm{Re}~a_{jk}+i\cdot\mathrm{Im}~a_{jk})\cdot \det(\mathrm{Re}~a_{jk}-i\cdot\mathrm{Im}~a_{jk})|\\
	&=|\det(A)|^2=\|\det(t_j(x_k))\|^2.
\end{align*}
Using this, we compute the volume of $P$ by
\begin{align}\label{equ_volume}
	V(P)=\frac{\|\det(t_j(x_k))\|^2}{\|\theta\|(\mathcal{N})^2\cdot \exp(\delta(X)/4)\cdot\prod_{j\neq k}^n G(x_j,x_k)}.
\end{align}
As $V(P)$ is independent of the choice of the poitns $x_1,\dots,x_n$, we may compute it by setting $x_j=y_j$ for all $1\le j\le n$. Then we get $\mathcal{N}=\mathcal{L}$ and the matrix $(t_j(y_k))$ is just a diagonal matrix such that
$$\|\det(t_j(y_k))\|=\prod_{j=1}^n\left(\|\theta\|(\mathcal{L})\cdot \prod_{k\neq j}G(y_k, y_j)\right).$$
Hence, we obtain for the volume of $P$
$$V(P)=\frac{\|\theta\|(\mathcal{L})^{2n-2}\prod_{j\neq k}^nG(y_j,y_k)}{\exp(\delta(X)/4)}.$$
If we apply this to Equation (\ref{equ_volume}), we obtain
\begin{align}\label{equ_fay}
	\|\theta\|\left(\mathcal{L}+\sum_{i=1}^n(y_i-x_i)\right)\|\theta\|(\mathcal{L})^{n-1}\prod_{j< k}^n G(x_j,x_k)G(y_j,y_k)=\|\det(t_j(x_k))\|
\end{align}
for any general choice of the points $x_1,\dots,x_n\in X$ as above. We remark that Fay \cite[Corollary 2.19]{Fay73} proved a holomorphic version of this equation for the Riemann $\theta$-function and the prime form $E(x,y)$. Note that Equation (\ref{equ_fay}) stays true for any choice of $y_1,\dots,y_n,x_1,\dots,x_n\in X$ and $\mathcal{L}\in J_{g-1}$ by continuity.

\section{Bound of the Arakelov--Green function}\label{sec_greenbound}
We prove Theorem \ref{thm_elkies} in this section. We continue the notation from the previous sections.
We would like to deduce an inequality involving the norm function $\|\theta\|$ and the Arakelov--Green function $G$ from (\ref{equ_fay}). For this purpose, we will use the following lemma, which is a consequence of Hadamard's inequality.
\begin{Lem}
	It holds
	$$\|\det(t_j(x_k))\|\le 
	\prod_{k=1}^n \sum_{j=1}^n \left(\|\theta\|(\mathcal{L}+y_j-x_k)\cdot \prod_{l\neq j}^n G(y_l, x_k)\right).$$
\end{Lem}
\begin{proof}
	Let $f_1,\dots,f_n\in\bigoplus_{j=1}^n\mathcal{M}[x_j]$ be an orthonormal basis of $\bigoplus_{j=1}^n\mathcal{M}[x_j]$ respecting its decomposition into its direct summands, that is $f_j\in \mathcal{M}[x_j]$ for any $j\le n$. Further, let $A=(a_{jk})\in \mathbb{C}^{n\times n}$ be the base change matrix such that $\phi(t_j)=\sum_{k=1}^n a_{jk}f_k$. By definition we have $\|\det(t_j(y_k))\|=|\det(A)|$. Note that $|a_{jk}|$ is the norm of $\phi(t_j)$ projected to the factor $\mathcal{M}[x_k]$ of $\bigoplus_{l=1}^n\mathcal{M}[x_l]$ and hence, it holds
	$$|a_{jk}|=\|t_j\|(x_k)=\|\theta\|(\mathcal{L}+y_j-x_k)\cdot \prod_{l\neq j}G(y_l,x_k).$$
	By a weak version of Hadamard's inequality we have $|\det(A)|\le \prod_{k=1}^n\sum_{j=1}^n|a_{jk}|$. Hence, we obtain the inequality in the lemma.
\end{proof}

Applying the lemma to Equation (\ref{equ_fay}), taking logarithm and integrating with respect to $\frac{\nu^g(\mathcal{L})}{g!}\mu(y_1)\dots\mu(y_n)$ we get

\begin{align*}
	&n H(X)+\sum_{j< k}^n g(x_j,x_k)\\
	&\le\sum_{k=1}^n \int_{X^n}\int_{J_{g-1}} \log \sum_{j=1}^n \|\theta\|(\mathcal{L}+y_j-x_k)\cdot\prod_{l\neq j}^n G(y_l,x_k)\tfrac{\nu^g(\mathcal{L})}{g!}\mu(y_1)\dots\mu(y_n).
\end{align*}
As $\int_{X^n}\log \prod_{l=1}^n G(y_l,x_k)\mu(y_1)\dots\mu(y_n)=0$ by the definition of the Arakelov--Green function in (\ref{equ_green}), we obtain by using the bound of the $\theta$-function in (\ref{bound_theta})
\begin{align}\label{equ_estimation}
	&n H(X)+\sum_{j< k}^n g(x_j,x_k)\\
	&\le-\rho nH(X)+nc_{g,\rho}+\sum_{k=1}^n \int_{X^n} \log \sum_{j=1}^n  G(y_j,x_k)^{-1}\mu(y_1)\dots\mu(y_n)\nonumber\\
	&\le-\rho nH(X)+nc_{g,\rho}+n\log n-\sum_{k=1}^n \int_{X^n} \min_{1\le j\le n}  g(y_j,x_k)\mu(y_1)\dots\mu(y_n).\nonumber
\end{align}
To bound the last term, let us in general define for any point $x\in X$
$$A_n(x)=-\int_{X^n}\min_{1\le j\le n} g(y_j, x)\mu(y_1)\dots\mu(y_n).$$
If we denote $M(y)=\{y'\in X~|~g(y',x)\ge g(y,x)\}\subseteq X$ for any $y\in X$ and $f(y)=\int_{M(y)}\mu$ for its volume, we can rewrite $A_n(x)$ as
\begin{align}\label{equ_An}
A_n(x)=-n\int_X g(y,x) f(y)^{n-1}\mu(y).
\end{align}
Note that
$$\int_X  f(y)^{n-1}\mu(y)=\frac{1}{n}$$
is the volume of the subset of $X^n$ consisting of the points $(y_1,\dots,y_n)\in X^n$ such that $g(y_j,x)$ is minimal for $j=1$, which has to be $1/n$ by symmetry.

Let $\{z_j\colon U_j\to \mathbb{C}\}_{j\in J}$ be a family of local coordinates on $X$ as in Theorem \ref{thm_elkies}. In particular, there are positive numbers $r_2>r_1>0$ with $r_2-r_1\le 1$ such that $X=\bigcup_{j\in J} z_j^{-1}(\mathbb{D}_{r_1})$ and $z_j(U_j)=\mathbb{D}_{r_2}$ for all $j\in J$. Let $C_1, C_2>0$ satisfy
$$C_2idz_jd\overline{z}_j\le \mu\le C_1idz_jd\overline{z}_j$$
on $U_j$ for all $j\in J$ and $C_2\le \frac{e^{4C_0}}{2\pi}$.

We fix a point $x\in X$ and a $j_0\in J$ such that $x\in z_{j_0}^{-1}(\mathbb{D}_{r_1})$. We define the coordinate $z(y)=z_{j_0}(y)-z_{j_0}(x)$ and set $U_x=z^{-1}(\mathbb{D}_{r_2-r_1})$. Now we discuss bounds for the Arakelov--Green function.
If $y\in U_{x}$, we get by definition of the constant $C_0$
$$|\log|z(y)|-g(y,x)|\le C_0.$$
Let us also give a lower bound for $g(y,x)$ if $y\notin U_{x}$. Since $g(y,x)$ is a superharmonic function in $y$, it does not have local minima on $X\setminus\{x\}$. Thus, there exists a path from $y$ to $x$, such that $g(\cdot,x)$ is decreasing on this path. This path intersects the loop $Z_{x}=\{w\in U_{j_0}~|~|z(w)|=r_2-r_1\}$ around $x$. In particular, we get
\begin{align*}
-g(y,x)&\le \sup_{w\in Z_{x}} -g(w,x)\\
&=\sup_{w\in Z_{x}} \left(\log|z(w)|-g(w,x)\right)-\log(r_2-r_1)\le C_0-\log(r_2-r_1).
\end{align*}

Our next goal is to find an upper bound for the function $f(y)$ in terms of the local coordinate $z(y)$ if $y\in U_x$. By the bound of the Arakelov--Green function we get the inclusions
\begin{align*}
X\setminus M(y)&=\{y'\in X~|~g(y',x)<g(y,x)\}\supseteq \{y'\in U_x~|~g(y',x)<g(y,x)\} \\
&\supseteq \{y'\in U_x~|~|z(y')|<|z(y)|e^{-2C_0}\}
\end{align*}
Thus,
\begin{align*}
	f(y)&\le 1-C_2\int_{\{y'\in U_x~|~|z(y')|<|z(y)|e^{-2C_0}\}} idzd\overline{z}=1- 2\pi C_2e^{-4C_0}|z(y)|^2.	
\end{align*}
Using Equation (\ref{equ_An}) we can now bound 
\begin{align}\label{equ_bound-A_n}
	A_n(x)\le C_0-\log(r_2-r_1)-nC_1\int_{U_x}\log|z(y)|(1-2\pi C_2e^{-4C_0}|z(y)|^2)^{n-1}idzd\overline{z}.
\end{align}
Next, we calculate the integral
\begin{align*}
	&\int_{U_x}\log|z(y)|(1-2\pi C_2e^{-4C_0}|z(y)|^2)^{n-1}idzd\overline{z}\\
	&=4\pi\int_{0}^{r_2-r_1}t(1-2\pi C_2e^{-4C_0}t^2)^{n-1}\log tdt\\
	&\ge 4\pi\sum_{j=0}^{n-1}\tbinom{n-1}{j}(-2\pi C_2e^{-4C_0})^j\int_{0}^{1} t^{2j+1}\log tdt\\
	&=-\pi\sum_{j=0}^{n-1}\tbinom{n-1}{j}(-2\pi C_2e^{-4C_0})^j\frac{1}{(j+1)^2}\\
	&=\frac{e^{4C_0}}{2C_2n}\sum_{j=1}^n\frac{\tbinom{n}{j}(-2\pi C_2e^{-4C_0})^{j}}{j},
\end{align*}
Note that the function 
$$a_n(x)=\sum_{j=1}^n\frac{\tbinom{n}{j}(-x)^j}{j}$$
has the derivative
$$a_n'(x)=\frac{1}{x}\sum_{j=1}^n\tbinom{n}{j}(-x)^j=\frac{1}{x}((1-x)^n-1).$$
Thus, $a_n(x)$ is decreasing for $0\le x\le 1$, such that $a_n(x)\ge a_n(1)=-\sum_{j=1}^n\frac{1}{j}$ for $x\in [0,1]$. As $\sum_{j=1}^n\frac{1}{j}\le 1+\int_1^n\frac{1}{x}dx=1+\log n$, we can bound
$$\sum_{j=1}^n\frac{\tbinom{n}{j}(-2\pi C_2e^{-4C_0})^{j}}{j}\ge -1-\log n^.$$
Thus, we conclude
\begin{align*}
	&\int_{U_x}\log|z(y)|(1-2\pi C_2e^{-4C_0}|z(y)|^2)^{n-1}idzd\overline{z}\ge -\frac{e^{4C_0}}{2C_2n}(1+\log n).
\end{align*}
If we apply this to inequality (\ref{equ_bound-A_n}), we finally get
$$A_n(x)\le C_0-\log(r_2-r_1)+\frac{C_1e^{4C_0}}{2C_2}(1+\log n).$$
Putting this into the computation in (\ref{equ_estimation}) we obtain
\begin{align*}
	&\sum_{j<k}^n g(x_j,x_k)\\
	&\le n\left(\left(\frac{C_1e^{4C_0}}{2C_2}+1\right)(\log n+1)-1+c_{g,\rho}-(1+\rho)H(X)+C_0-\log(r_2-r_1)\right).
\end{align*}
If we set $\rho=1/g$ and use $c_{g,1/g}\le \frac{3}{2}g\log g+4$, we obtain 
\begin{align*}
	&\sum_{j<k}^n g(x_j,x_k)\\
	&\le n\left(\left(\frac{C_1e^{4C_0}}{2C_2}+1\right)(\log n+1)+\tfrac{3}{2}g\log g+3-\tfrac{g+1}{g}H(X)+C_0-\log(r_2-r_1)\right).
\end{align*}
Thus, we have proved Theorem \ref{thm_elkies}. Corollary \ref{cor_elkies} follows directly from the arguments in the introduction.


\begin{thebibliography}{88}
	\bibitem{Ara74} Arakelov, S. Y.: \emph{Intersection theory of divisors on an arithmetic surface}. Izv. Akad. USSR \textbf{8}, no. 6, 1167--1180 (1974).
	\bibitem{BP05} Baker, M. and Petsche, C.: \emph{Global discrepancy and small points on elliptic curves}. Int. Math. Res. Not. {\bf 2005}, no. 61, 3791--3834.
	\bibitem{BR07} Baker, M. and Rumely, R.: \emph{Harmonic analysis on metrized graphs}. Canad. J. Math. {\bf 59} (2007), no. 2, 225--275.
	\bibitem{EdJ11} Edixhoven, B; de Jong, R.: \emph{Bounds for Arakelov invariants of modular curves.} In: Computational aspects of modular forms and Galois representations, 217--256, Ann. of Math. Stud., \textbf{176}, Princeton Univ. Press, Princeton, NJ, 2011. 
	\bibitem{Fal84} Faltings, G.: \emph{Calculus on arithmetic surfaces}. Ann. of Math. \textbf{119}, 387--424 (1984).
	\bibitem{Fay73} Fay, J. D.: \emph{Theta functions on Riemann surfaces}. Lecture Notes in Mathematics, Vol. \textbf{352}. Springer-Verlag, Berlin-New York, 1973.
	\bibitem{Jav14} Javanpeykar, A.: \emph{Polynomial bounds for Arakelov invariants of Belyi curves. With an appendix by Peter Bruin.} Algebra Number Theory \textbf{8} (2014), no. 1, 89--140.
	\bibitem{Lan88} Lang, S.: \emph{Introduction to Arakelov theory}. Springer-Verlag, New York, 1988.
	\bibitem{LSW21} Looper, N.; Silverman, J. H. and Wilms, R.: \emph{A uniform quantitative Manin--Mumford theorem for curves over function fields}. Preprint, arXiv:2101.11593 (2021).
	\bibitem{Mer11} Merkl, F.: \emph{An upper bound for Green functions on Riemann surfaces.} In: Computational aspects of modular forms and Galois representations, 203--215, Ann. of Math. Stud., \textbf{176}, Princeton Univ. Press, Princeton, NJ, 2011. 
	\bibitem{Par18} Parent, P.: \emph{Heights on squares of modular curves. With an appendix by Pascal Autissier.} Algebra Number Theory \textbf{12} (2018), no. 9, 2065--2122.
	\bibitem{Par68} Par{\v{s}}in, A. N.: \emph{Algebraic curves over function fields. I}. Math. USSR-Izvestiya \textbf{2} (1968), no. 5, 1145--1170.
	\bibitem{Wil17} Wilms, R.: \emph{New explicit formulas for Faltings' delta-invariant}.  Invent. Math. \textbf{209} (2017), no. 2, 481--539.
\end{thebibliography}
\end{document}